\documentclass[11pt,reqno]{amsart}

\usepackage{amscd,amssymb,amsmath,amsthm}
\usepackage[arrow,matrix]{xy}
\usepackage{graphicx}
\usepackage{epstopdf}
\usepackage{color}
\newtheorem{thm}{Theorem}
\newtheorem{defn}{Definition}
\newtheorem{lemma}{Lemma}
\newtheorem{pro}{Proposition}
\newtheorem{rk}{Remark}
\newtheorem{cor}{Corollary}

\newtheorem{ex}{Example}

\numberwithin{equation}{section} 

\def\r{\rho}

\def\Q{\mathbb Q}
\def\C{\mathbb C}

\def\C{\mathbb{C}}
\def\N{\mathbb N}
\def\Z{\mathbb{Z}}

\begin{document}

\title[$p$-adic dynamical systems of a $(3,1)$-rational function]{$p$-adic dynamical systems of $(3,1)$-rational functions with unique fixed
point}

\author{A.R. Luna, U.A. Rozikov, I.A. Sattarov}
\address{A.\ R. \ Luna \\ Department of Mathematics, California State University, Fullerton, CA 92831}\email{lunaalex1998@gmail.com}

\address{U.\ A.\ Rozikov \\ Institute of mathematics,
81, Mirzo Ulug'bek str., 100125, Tashkent, Uzbekistan.} \email
{rozikovu@yandex.ru}

 \address{I.\ A.\ Sattarov \\ Institute of mathematics,
81, Mirzo Ulug'bek str., 100125, Tashkent, Uzbekistan.} \email
{sattarovi-a@yandex.ru}

\begin{abstract}
We describe the set of all $(3,1)$-rational functions given on the set of complex
$p$-adic field $\C_p$ and having a unique fixed point. We study
$p$-adic dynamical systems generated by such $(3,1)$-rational functions and
show that the fixed point is indifferent and therefore
the convergence of the trajectories is not the typical
case for the dynamical systems. We obtain Siegel disks of these dynamical systems.
 Moreover an upper bound for the set of limit points of each trajectory is given.
 For each $(3,1)$-rational function on $\C_p$ there is a point
 $\hat x=\hat x(f)\in \C_p$ which is
 zero in its denominator. We give explicit formulas of radii of spheres
(with the center at the fixed point) containing some points that
the trajectories (under actions of $f$) of the points after a finite step
come to $\hat x$. For a class of $(3,1)$-rational functions defined on the set of $p$-adic numbers $\Q_p$
we study ergodicity properties of the corresponding dynamical systems.
We show that if $p\geq 3$ then
the $p$-adic dynamical system reduced on each invariant sphere is not ergodic with respect to Haar measure.
For $p=2$, under some conditions we prove non ergodicity and
show that there exists a sphere on which the dynamical system is ergodic.
Finally, we give a characterization of periodic orbits and some uniformly local properties of the $(3.1)-$rational functions.
\end{abstract}

\keywords{Rational dynamical systems; fixed point; invariant set; Siegel disk;
complex $p$-adic field; ergodic.} \subjclass[2010]{46S10, 12J12, 11S99,
30D05, 54H20.} \maketitle

\section{Introduction}
 A function is called a $(n,m)$-rational function if and only if it
can be written in the form $f(x)={P_n(x)\over Q_m(x)}$, where
$P_n(x)$ and $Q_m(x)$ are polynomial functions with degree $n$ and
$m$ respectively, and $Q_m(x)$ is not the zero polynomial.

In this paper we study dynamical systems generated by a $(3,1)$-rational
function. This paper is a continuation of works \cite{ARS}, \cite{KMa}, \cite{M1}, \cite{UF},
\cite{RS}, \cite{RS2} and \cite{S}. Therefore for motivations of the consideration of such dynamical
systems and history of the problem we refer to these papers and the references therein.

The paper is organized as follows: in
Section 2 we give some preliminaries. Section 3 contains the
definition of the $(3,1)$-rational function and main results about behavior of trajectories of
the $p$-adic dynamical system.  Siegel
disks of these dynamical systems are studied.
An upper bound for the set of limit points of each trajectory is given.
  In Section 4 for a class of $(3,1)$-rational functions
we study ergodicity properties of the dynamical systems on the set of $p$-adic numbers $\Q_p$.
For each such function we describe all possible invariant spheres and study
ergodicity of each $p$-adic dynamical system with respect to Haar measure
reduced on each invariant sphere. For $p\geq 3$ it is proved that the dynamical
systems are not ergodic. But for $p=2$ under some conditions the dynamical system may be ergodic.
The last section contains results related to some uniformly local properties and periodic orbits of the rational function.

\section{Preliminaries}

\subsection{$p$-adic numbers}

Let $\Q$ be the field of rational numbers. The greatest common
divisor of the positive integers $n$ and $m$ is denoted by
$(n,m)$. Every rational number $x\neq 0$ can be represented in the
form $x=p^r\frac{n}{m}$, where $r,n\in\mathbb{Z}$, $m$ is a
positive integer, $(p,n)=1$, $(p,m)=1$ and $p$ is a fixed prime
number.

The $p$-adic norm of $x$ is given by
$$
|x|_p=\left\{
\begin{array}{ll}
p^{-r}, & \ \textrm{ for $x\neq 0$},\\[2mm]
0, &\ \textrm{ for $x=0$}.\\
\end{array}
\right.
$$
It has the following properties:

1) $|x|_p\geq 0$ and $|x|_p=0$ if and only if $x=0$,

2) $|xy|_p=|x|_p|y|_p$,

3) the strong triangle inequality
$$
|x+y|_p\leq\max\{|x|_p,|y|_p\},
$$

3.1) if $|x|_p\neq |y|_p$ then $|x+y|_p=\max\{|x|_p,|y|_p\}$,

3.2) if $|x|_p=|y|_p$ then $|x+y|_p\leq |x|_p$,

this is a non-Archimedean one.

The completion of $\Q$ with respect to the $p$-adic norm defines the
$p$-adic field which is denoted by $\Q_p$ (see \cite{Ko}).

The algebraic completion of $\Q_p$ is denoted by $\C_p$ and it is
called {\it complex $p$-adic numbers}.  For any $a\in\C_p$ and
$r>0$ denote
$$
U_r(a)=\{x\in\C_p : |x-a|_p<r\},\ \ V_r(a)=\{x\in\C_p :
|x-a|_p\leq r\}
$$
and
$$
S_r(a)=\{x\in\C_p : |x-a|_p= r\}.
$$

A function $f:U_r(a)\to\C_p$ is said to be {\it analytic} if it
can be represented by
$$
f(x)=\sum_{n=0}^{\infty}f_n(x-a)^n, \ \ \ f_n\in \C_p,
$$ which converges uniformly on the ball $U_r(a)$.

\subsection{Dynamical systems in $\C_p$}

Recall some known facts concerning dynamical
systems $(f,U)$ in $\C_p$, where $f: x\in U\to f(x)\in U$ is an
analytic function and $U=U_r(a)$ or $\C_p$ (see for example \cite{PJS}).

Now let $f:U\to U$ be an analytic function. Denote
$f^n(x)=\underbrace{f\circ\dots\circ f}_n(x)$.

If $f(x_0)=x_0$ then $x_0$
is called a {\it fixed point}. The set of all fixed points of $f$
is denoted by Fix$(f)$. A fixed point $x_0$ is called an {\it
attractor} if there exists a neighborhood $U(x_0)$ of $x_0$ such
that for all points $x\in U(x_0)$ it holds
$\lim\limits_{n\to\infty}f^n(x)=x_0$. If $x_0$ is an attractor
then its {\it basin of attraction} is
$$
A(x_0)=\{x\in \C_p :\ f^n(x)\to x_0, \ n\to\infty\}.
$$
A fixed point $x_0$ is called {\it repeller} if there  exists a
neighborhood $U(x_0)$ of $x_0$ such that $|f(x)-x_0|_p>|x-x_0|_p$
for $x\in U(x_0)$, $x\neq x_0$.

Let $x_0$ be a fixed point of a
function $f(x)$.
Put $\lambda=f'(x_0)$. The point $x_0$ is attractive if $0<|\lambda|_p < 1$, {\it indifferent} if $|\lambda|_p = 1$,
and repelling if $|\lambda|_p > 1$.

The ball $U_r(x_0)$ (contained in $V$) is said to
be a {\it Siegel disk} if each sphere $S_{\r}(x_0)$, $\r<r$ is an
invariant sphere of $f(x)$, i.e. if $x\in S_{\r}(x_0)$ then all
iterated points $f^n(x)\in S_{\r}(x_0)$ for all $n=1,2\dots$.  The
union of all Siegel desks with the center at $x_0$ is said to {\it
a maximum Siegel disk} and is denoted by $SI(x_0)$.

Two maps $f:U\rightarrow U$ and $g:V\rightarrow V$ are said
to be {\it topologically conjugate} if there exists a homeomorphism $h: U \rightarrow V$ such
that $h \circ f = g \circ h$. The homeomorphism $h$ is called a {\it topological conjugacy}.
Mappings which are topologically conjugate are completely equivalent in
terms of their dynamics. For example, if $f$ is topologically conjugate to $g$ via
$h$, and $x_0$ is a fixed point for $f$, then $h(x_0)$ is fixed for $g$. Indeed, $h(x_0) = h(f(x_0)) =
g(h(x_0))$.

\section{$(3,1)$-Rational $p$-adic dynamical systems}

In this paper we consider the dynamical system associated with the
$(3,1)$-rational function $f:\C_p\to\C_p$ defined by
\begin{equation}\label{f}
f(x)=\frac{x^3+ax^2+bx+c}{dx+e}, \ \ d\neq 0,\ \ \ a,b,c,d,e\in \C_p.
\end{equation}
where  $x\neq \hat x=-\frac{e}{d}$.

\subsection{Functions with unique fixed point} For $(3,1)$-rational functions with representation of (\ref{f}),  the equation
$f(x)=x$ of fixed points is equivalent to  the following cubic equation
\begin{equation}\label{ce}
x^3+(a-d)x^2+(b-e)x+c=0.
\end{equation}
Since $\C_p$ is algebraically closed the equation (\ref{ce}) may have
three solutions with one of the following relations:

(i) One solution having multiplicity three;

(ii) Two solutions, one of which has multiplicity two;

(iii) Three distinct solutions.

In this paper we investigate the behavior of trajectories of an
arbitrary $(3,1)$-rational dynamical system in complex $p$-adic
filed $\C_p$ when there is unique fixed point for $f$, i.e., we
consider case (i).

Denote $x_0$ to be the unique solution of the equation (\ref{ce}), hence $x_0$ has multiplicity three.
Then we have $x^3+(a-d)x^2+(b-e)x+c=(x-x_0)^3$ and

\begin{equation}\label{c3}
\left\{\begin{array}{lll}
3x_0=d-a \\[2mm]
3x_0^2=b-e \\[2mm]
x_0^3=-c.
\end{array}
\right.
\end{equation}

The following proposition characterizes all $(3,1)$-rational functions with a unique fixed point.

\begin{pro}\label{pro1}
Any $(3,1)$-rational function with unique fixed point is topologically conjugate to a function of the following form
\begin{equation}\label{fa}
f(x)=\frac{x^3+ax^2+bx}{ax+b}, \ \ ab\neq 0, \ \  a,b\in \C_p.
\end{equation}
where $x\neq \hat x=-\frac{b}{a}$.
\end{pro}
\begin{proof}
Consider a homeomorphism $h: \C_p \rightarrow \C_p$ defined by $x=h(t)=t+x_0$. So $h^{-1}(t)=t-x_0$.
Note that the function $f$ is topologically conjugate to the function $h^{-1}\circ f\circ h$.
We have
\begin{equation}\label{fb}
(h^{-1}\circ f\circ h)(t)=\frac{t^3+At^2+Bt+C}{Dt+E},
\end{equation}
where $A=3x_0+a$, $B=3x_0^2+(2a-d)x_0+b$, $C=x_0^3+(a-d)x_0^2+(b-e)x_0+c$, $D=d$ and $E=dx_0+e$.

It should be noted $A=D$, $B=E$, $C=0$ and $AB\neq 0$.
Indeed, from (\ref{c3}) we have
$$3x_0=d-a \ \ {\rm and} \ \  A=3x_0+a=d=D.$$
Also, $B-E=3x_0^2+2(a-d)x_0+b-e=0$, i.e. $B=E$ and $C=x_0^3+(a-d)x_0^2+(b-e)x_0+c=0$.
Note that $D=d\neq 0$ and $E=dx_0+e\neq 0$. Consequently, $DE=AB\neq 0$.
Thus
$$(h^{-1}\circ f\circ h)(t)=\frac{t^3+At^2+Bt}{At+B}, \ \ AB\neq 0.$$
\end{proof}

It is easy to see that the function (\ref{fa}) has a unique fixed point at $x_0=0$.
For (\ref{fa}) we have
$$f'(x_0)=f'(0)=1,$$
i.e.,  the point $x_0$ is an indifferent point for (\ref{fa}).

For any $x\in \C_p$, $x\ne \hat x$, by simple calculations we
get
\begin{equation}\label{f2}
    |f(x)|_p=|x|_p\cdot{|(x-x_1)(x-x_2)|_p\over {|ax+b|_p}},
\end{equation}
where $x_{1,2}={{-a\pm\sqrt{a^2-4b}}\over{2}}$.
Denote
\begin{equation}\label{P}
\mathcal P=\{x\in \C_p: \exists n\in \N\cup\{0\}, f^n(x)=\hat x\},
\end{equation}
$$\delta=|a|_p, \ \ \alpha=|x_1|_p \ \ {\rm and} \ \ \beta=|x_2|_p.$$
The following proposition gives relations between these real numbers.
\begin{pro}\label{1} We have
\begin{itemize}
\item[1.] If $\alpha\neq\beta$, then $\delta=\max\{\alpha,\beta\}$.
\item[2.] If $\alpha=\beta$, then $\delta\leq\alpha=\beta$.
\item[3.] $|b|_p=\alpha\beta$.
 \end{itemize}
\end{pro}

\begin{proof}
\begin{itemize}
\item[1.] Since $\alpha\neq\beta$, by properties of $p$-adic norms we have
$$\delta=|a|_p=|-a|_p=|x_1+x_2|_p=\max\{|x_1|_p,|x_2|_p\}=\max\{\alpha,\beta\}.$$
\item[2.] Since $\alpha=\beta$, by properties of $p$-adic norms, we have $\delta\leq\alpha=\beta$.
\item[3.] By Vieta's formulas, since $-a=x_1+x_2$, we have that $b=x_1x_2$. It follows that $|b|_p=|x_1x_2|_p=|x_1|_p|x_2|_p=\alpha\beta$.
\end{itemize}
\end{proof}

\begin{rk}\label{r1}
By the symmetry of $x-x_1$ and $x-x_2$ in (\ref{f2}), we only consider the dynamical
system $(f,\mathbb{C}_p\setminus\mathcal P)$ for the cases $\alpha<\beta$ and $\alpha=\beta$.
\end{rk}

\subsection{Case: $\alpha<\beta$}

By Proposition \ref{1}, the only possible case is $\alpha<\beta=\delta$.
Now for $\alpha<\beta=\delta$, define the function $\phi_{\alpha,\beta}:[0,+\infty)\rightarrow[0,+\infty)$ by

$$\phi_{\alpha,\beta}(r)=\left\{\begin{array}{lllll}
r, \ \ {\rm if} \ \ r<\alpha\\[2mm]
\hat\alpha, \ \ {\rm if} \ \ r=\alpha\\[2mm]
r, \ \ {\rm if} \ \ \alpha<r<\beta\\[2mm]
\hat\beta, \ \ {\rm if} \ \ r=\beta\\[2mm]
\frac{r^2}{\beta}, \ \ \ \ {\rm if} \ \ r>\beta
\end{array}
\right.
$$
where $\hat\alpha$ and $\hat\beta$ are some positive numbers with $\hat\alpha>0$ and $\hat\beta\leq\beta$.\\
Now, using (\ref{f2}) we get the following lemma:

\begin{lemma}\label{lf1}
If $\alpha<\beta$ and $x\in{S}_r(0)$, then the following formula holds for the function (\ref{fa}):
$$|f^n(x)|_p=\phi_{\alpha,\beta}^n(r).$$
\end{lemma}
We now see that the real dynamical system $\phi_{\alpha,\beta}^n$ is directly related to the $p$-adic dynamical system $f^n(x), \ n\geq{1}, \ x\in\mathbb{C}_p\setminus\mathcal P$. The following lemma gives properties to this real dynamical system.

\begin{lemma}\label{l1} If $\alpha<\beta=\delta$, then the dynamical
system generated by $\phi_{\alpha,\beta}(r)$ has
the following properties:
\begin{itemize}
\item[1.] ${\rm Fix}(\phi_{\alpha,\beta})=\{r: 0\leq r<\alpha\}\cup\{\alpha:\, if \,
\hat\alpha=\alpha\}\cup\{r: \alpha<r<\beta\}\cup\{\beta: \, if \, \hat\beta=\beta\}$.
\item[2.] If $r>\beta$, then
$$\lim_{n\to\infty}\phi_{\alpha,\beta}^n(r)=+\infty.$$
\item[3.] Let $r=\alpha$.
\begin{itemize}
\item[3.1)] If $\hat\alpha<\alpha$ or $\alpha<\hat\alpha<\beta$, then
$\phi_{\alpha,\beta}^n(r)=\hat\alpha$ for any $n\geq 1$.
\item[3.2)] If $\hat\alpha=\beta$, then $\phi_{\alpha,\beta}(\alpha)=\beta$.
\item[3.3)] If $\hat\alpha>\beta$, then $$\lim_{n\to\infty}\phi_{\alpha,\beta}^n(\alpha)=+\infty.$$
\end{itemize}
\item[4.] Let $r=\beta$.
\begin{itemize}
\item[4.1)] If $\hat\beta<\alpha$ or $\alpha<\hat\beta<\beta$, then
$\phi_{\alpha,\beta}^n(r)=\hat\beta$ for any $n\geq 1$.
\item[4.2)] If $\hat\beta=\alpha$, then $\phi_{\alpha,\beta}(\beta)=\alpha$.
\end{itemize}
\end{itemize}
\end{lemma}

\begin{proof} 1. This is the result of a simple analysis
of the equation $\phi_{\alpha,\beta}(r)=r$.

2. By definition of the function $\phi_{\alpha,\beta}$, we have
$\phi_{\alpha,\beta}(r)={{r^2}\over{\beta}}>r$ for any $r>\beta$.

So $\phi_{\alpha,\beta}^n(r)=\beta\left({r\over\beta}\right)^{2^n}$.
Consequently,
$$\lim_{n\to\infty}\phi_{\alpha,\beta}^n(r)=+\infty.$$

Parts 3 and 4 easily follow from part 1,  part 2 and the definition of the function $\phi_{\alpha,\beta}$.
\end{proof}

Applying the lemma, we now obtain a theorem that describes the $p$-adic dynamical system generated by the function (\ref{fa}).\\
For $\alpha<\beta$, denote the following:\\
$$\hat\alpha(x)=|f(x)|_p, \ \ {\rm if} \ \ x\in{S}_{\alpha}(0)$$
$$\hat\beta(x)=|f(x)|_p, \ \ {\rm if} \ \ x\in{S}_{\beta}(0).$$
Using Lemmas \ref{lf1} and \ref{l1}, we obtain the following:

\begin{thm}\label{t1} If $\alpha<\beta=\delta$ and $x\in S_r(x_0)$, then
 the $p$-adic dynamical system generated by the function (\ref{fa}) has the following properties:
\begin{itemize}
\item[1.]
\begin{itemize}
\item[1.1)] $SI(x_0)=U_{\alpha}(0)$.
\item[1.2)] If $\alpha<r<\beta$, then $S_r(0)$ is an invariant for $f$.
\end{itemize}
\item[2.] If $r>\beta$, then $$\lim_{n\to\infty}|f^n(x)|_p=+\infty.$$
\item[3.] Let $r=\alpha$ and $x\in S_{\alpha}(0)\setminus\mathcal P$.
\begin{itemize}
\item[3.1)] If $\hat\alpha(x)=\alpha$, then $f(x)\in S_{\alpha}(0)$.
\item[3.2)] If $\hat\alpha(x)<\beta$ and $\hat\alpha(x)\ne\alpha$, then $f^n(x)\in S_{\hat\alpha(x)}(0)$ for any $n\geq 1$.
\item[3.3)] If $\hat\alpha(x)=\beta$, then $f(x)\in S_{\beta}(0)$.
\item[3.4)] If $\hat\alpha(x)>\beta$, then $$\lim_{n\to\infty}|f^n(x)|_p=+\infty.$$
\end{itemize}
\item[4.] Let $r=\beta$ and $x\in S_{\beta}(0)\setminus\mathcal P$.
\begin{itemize}
\item[4.1)] If $\hat\beta(x)=\beta$, then $f(x)\in S_{\beta}(0)$.
\item[4.2)] If $\hat\beta(x)<\beta$ and $\hat\beta(x)\ne\alpha$, then $f^n(x)\in S_{\hat\beta(x)}(0)$ for any $n\geq 1$.
\item[4.3)] If $\hat\beta(x)=\alpha$, then $f(x)\in S_{\alpha}(0)$.
\end{itemize}
\item[5.] $\mathcal P\subset S_{\alpha}(0)\cup S_{\beta}(0)$.
\end{itemize}
\end{thm}
\begin{proof}
1. By Lemma \ref{lf1} and part 1 of Lemma \ref{l1} we see that the
 spheres $S_r(0)$ are invariant for $f$ if $r<\alpha$ or $\alpha<r<\beta$.
 Consequently, $U_{\alpha}(0)\subset SI(x_0)$.

 Moreover, by part 3 of Lemma \ref{l1} we know that $S_r(0)$ is not invariant of $f$ for $r=\alpha$. Thus $SI(x_0)=U_{\alpha}(0)$.

2. Follows from Lemma \ref{lf1} and part 2 of Lemma
\ref{l1}.

Parts 3 and 4 follow from the definition of the function $\phi_{\alpha,\beta}$  as well as parts 1 and 2 of this Theorem.

5. In this case we have $|\hat x|_p=\alpha$. From previous parts of this Theorem if $|x|_p\not\in\{\alpha,\beta\}$, then $|f(x)|_p\neq\alpha$. Consequently, $\mathcal P\subset S_{\alpha}(0)\cup S_{\beta}(0)$.
\end{proof}

\subsection{Case: $\alpha=\beta$}

By proposition \ref{1}, the only possible cases are $\delta<\alpha$ and $\delta=\alpha$.

Using the formula (\ref{f2}) we construct the following functions:

For $\delta<\alpha$, define the function $\zeta_{\alpha,\delta}:[0,+\infty)\rightarrow[0,+\infty)$ by

$$\zeta_{\alpha,\delta}(r)=\left\{\begin{array}{lllll}
r, \ \ {\rm if} \ \ r<\alpha\\[2mm]
\alpha', \ \ {\rm if} \ \ r=\alpha\\[2mm]
\frac{r^3}{\alpha^2},\ \ \ {\rm if} \ \ \alpha<r<\frac{\alpha^2}{\delta}\\[2mm]
\delta', \ \ {\rm if} \ \ r=\frac{\alpha^2}{\delta}\\[2mm]
\frac{r^2}{\delta}, \ \ \ \ {\rm if} \ \ r>\frac{\alpha^2}{\delta}
\end{array}
\right.
$$
where $\alpha'$ and $\delta'$ are positive numbers with $\alpha'\leq\alpha$ and $\delta'\geq\frac{\alpha^4}{\delta^3}$.

For $\delta=\alpha$, define the function $\eta_{\alpha}:[0,+\infty)\rightarrow[0,+\infty)$ by
$$\eta_{\alpha}(r)=\left\{\begin{array}{lllll}
r, \ \ {\rm if} \ \ r<\alpha\\[2mm]
\bar\alpha, \ \ {\rm if} \ \ r=\alpha\\[2mm]
\frac{r^2}{\alpha},\ \ \ {\rm if} \ \ \alpha<r
\end{array}
\right.
$$
where $\bar\alpha$ is a positive number with $\bar\alpha>0$.

\begin{lemma}\label{lf2}
If $\alpha=\beta$ and $x\in{S}_{r}(0)$, then the following formula holds for the function (\ref{fa})
$$|f^n(x)|_p=\left\{\begin{array}{lllll}
\zeta_{\alpha,\delta}^n(r), \ \ {\rm if} \ \ \delta<\alpha\\[2mm]
\eta_{\alpha}^n(r), \ \ {\rm if} \ \ \delta=\alpha.
\end{array}
\right.
$$
\end{lemma}
Thus the real dynamical systems $\zeta_{\alpha,\delta}^n$ and $\eta_{\alpha}^n$ are directly related to the $p$-adic dynamical system $f^n(x), \ n\geq{1}, \ x\in\mathbb{C}_p\setminus\mathcal P$. The following lemmas gives properties to these real dynamical systems.

\begin{lemma}\label{l2}
If $\delta<\alpha$, then the dynamical system generated by $\zeta_{\alpha,\delta}$ has the following properties:
\begin{itemize}
\item[1.] ${\rm Fix}(\zeta_{\alpha,\delta})=\{r:0\leq{r}<\alpha\}\cup\{\alpha:\alpha'=\alpha\}$.
\item[2.] If $r=\alpha$ and $\alpha'<\alpha$, then $\zeta_{\alpha,\delta}^n(r)=\alpha'$ for all $n\geq{1}.$
\item[3.] If $r>\alpha$, then $\lim\limits_{n\rightarrow\infty}\zeta_{\alpha,\delta}^n(r)=+\infty$.
\end{itemize}
\end{lemma}
\begin{proof}
\begin{itemize}
\item[1.] This is a simple observation of the function $\zeta_{\alpha,\delta}$.
\item[2.] If $r=\alpha$ and $\alpha'<\alpha$, then $\zeta_{\alpha,\beta}(\alpha)=\alpha'$ and $\alpha'$ is a fixed point for $f$. Consequently, $\zeta_{\alpha,\delta}^n(r)=\alpha'$ for all $n\geq{1}$.
\item[3.] If $r\in(\alpha,\frac{\alpha^2}{\delta}]$ then there exists an $n_0$ such that
$$\zeta_{\alpha,\delta}^{n_0}(r)>\dfrac{\alpha^2}{\delta}.$$
If $r>\frac{\alpha^2}{\delta},$ then by definition of the function $\zeta_{\alpha,\delta}$, we have
$$\zeta_{\alpha,\delta}(r)=\dfrac{r^2}{\delta} \ \ {\rm and} \ \ \zeta_{\alpha,\delta}^{n}(r)=\dfrac{r^{2^{n}}}{\delta^{2^{n}-1}}.$$
Since $\frac{r}{\delta}>1$, then
$$\lim\limits_{n\rightarrow\infty}\zeta_{\alpha,\delta}^n(r)=\lim\limits_{n\rightarrow\infty}(\dfrac{r}{\delta})^{2^n}\delta=+\infty.$$
It follows that
$$\lim\limits_{n\rightarrow\infty}\zeta_{\alpha,\delta}^n(r)=+\infty,$$
for all $r>\alpha$.
\end{itemize}
\end{proof}

\begin{lemma}\label{l3}
If $\delta=\alpha$, then the dynamical system generated by $\eta_{\alpha}$ has the following properties:
\begin{itemize}
\item[1.] ${\rm Fix}(\eta_{\alpha})=\{r:0\leq{r}<\alpha\}\cup\{\alpha:\bar\alpha=\alpha\}.$
\item[2.] If $r=\alpha$ and $\bar\alpha<\alpha$, then $\eta_{\alpha}^n(r)=\bar\alpha$ for all $n\geq{1}$.
\item[3.] If $r\geq{\alpha}$ and $\bar\alpha>\alpha$, then $\lim\limits_{n\rightarrow\infty}\eta_{\alpha}^n(r)=+\infty$.
\end{itemize}
\end{lemma}

\begin{proof}
\begin{itemize}
\item[1.] This is a simple observation of the function $\eta_{\alpha}$.
\item[2.] Since $r=\alpha$, $\eta_{\alpha}(\alpha)=\bar\alpha$. Since $\bar\alpha<\alpha$, by part 1, $\eta_{\alpha}^n(\alpha)=\bar\alpha$ for all $n\geq{1}$.
\item[3.] Since $\eta_{\alpha}$ is not bounded above, it suffices to show $\eta_{\alpha}^n(r)<\eta_{\alpha}^{n+1}(r)$ for all $r\geq{\alpha}$, where $\bar\alpha>\alpha$.
\\
Case I: Suppose $r=\alpha$. Then $\eta_{\alpha}(\alpha)=\bar\alpha,$ and $\bar\alpha>\alpha.$\\
Case II: Suppose $r>\alpha.$ Then, it can easily be shown that
$$\eta_{\alpha}^n(r)=\dfrac{r^{2^n}}{\alpha^{2^n-1}}\Rightarrow\eta_{\alpha}^{n+1}(r)=\dfrac{r^{2^{n+1}}}{\alpha^{2^{n+1}-1}}.$$
Now, since $\frac{r}{\alpha}>1$,
$$\dfrac{r^{2^n}}{\alpha^{2^n}}<\dfrac{r^{2^{n+1}}}{\alpha^{2^{n+1}}}
\Rightarrow\dfrac{r^{2^n}}{\alpha^{2^n-1}}<\dfrac{r^{2^{n+1}}}{\alpha^{2^{n+1}-1}}\Rightarrow\eta_{\alpha}^n(r)<\eta_{\alpha}^{n+1}(r).$$
It now follows that $$\lim\limits_{n\rightarrow\infty}\eta_{\alpha}^n(r)=+\infty.$$
\end{itemize}
\end{proof}
Applying the lemmas, we now describe the $p$-adic dynamical system generated by the function (\ref{fa}).

For $\delta<\alpha$, denote the following:
\begin{equation}\label{*}
\alpha^*(x)=|f(x)|_p, \ \ {\rm if} \ \ x\in{S}_{\alpha}(0).
\end{equation}
Using Lemmas \ref{lf2} and \ref{l2}, we obtain the following:

\begin{thm}\label{t2}
For $\delta<\alpha$ and $x\in{S}_{r}(0)\setminus\mathcal P$, we have
\begin{itemize}
\item[1.] $SI(x_0)=U_{\alpha}(0)$.
\item[2.] If $x\in{S}_{\alpha}(0)$, then:
\begin{itemize}
\item[a.] If $\alpha^*(x)=\alpha$, then $f(x)\in{S}_{\alpha}(0)$.
\item[b.] If $\alpha^*(x)<\alpha$, then $f^n(x)\in{S}_{\alpha^*(x)}(0)$ for all $n\geq{1}$.
\end{itemize}
\item[3.] If $r>\alpha$, then $\lim\limits_{n\rightarrow\infty}|f^n(x)|_p=+\infty$.
\end{itemize}
\end{thm}

\begin{proof} 1. By Lemma \ref{lf2} and parts 2 and 3 of Lemma \ref{l2}, we know that $S_r(0)$ is not an invariant of $f$ for $r\geq\alpha$. Consequently, $SI(x_0)\subset U_{\alpha}(0)$.

By Lemma \ref{lf2} and part 1 of  Lemma \ref{l2} if
$r<\alpha$ and $x\in S_r(0)$ then
$|f^n(x)|_p=\zeta^n_{\alpha,\delta}(r)=r$, i.e., $f^n(x)\in
S_r(0)$. Hence $U_{\alpha}(0)\subset SI(x_0)$ and thus
$SI(x_0)=U_{\alpha}(0).$

2. Let $x\in S_\alpha(0)$.

If $\alpha^*(x)=\alpha$, then by (\ref{*}) we have
$|f(x)|_p=\alpha$, i.e., $f(x)\in S_{\alpha}(0)$.

By part 1 of this Theorem if $\alpha^*(x)<\alpha$, then $f^n(x)\in
S_{\alpha^*(x)}(0)$ for any $n\geq 1$.

Part 3 easily follows from Lemma \ref{lf2} and part 3 of Lemma \ref{l2}.
\end{proof}

Realizing that (\ref{P}) has the form
$$\mathcal P=\bigcup_{k=0}^{\infty}\mathcal P_k, \  {\rm such \ that} \ \mathcal P_k=\{x\in\mathbb{C}_p:f^k(x)=\hat{x}\},$$
we obtain the following Theorem.

\begin{thm}
\begin{itemize}
\item[1.] $\mathcal P_k\neq\emptyset$ for any $k=0,1,2,...$
\item[2.] If $\delta<\alpha=\beta$, then $\mathcal P_k\subset S_{r_k}(0)$,   where
$r_k=\alpha\cdot\left({\alpha\over\delta}\right)^{1\over{3^k}}$,
$k=0,1,2,... \,.$
\end{itemize}
\end{thm}

\begin{proof}
\begin{itemize}
\item[1.] We proceed to show that $\mathcal P_k\neq\emptyset$ by induction.
Cleary, $\mathcal P_0\neq\emptyset$ since $\hat{x}\in\mathbb{C}_p$.
Suppose now that $\mathcal P_k\neq\emptyset.$ That is, there exists a $y\in\mathbb{C}_p$ such that
$$f^k(y)=\hat{x}.$$
We need to show that there exists an $x_0\in\mathbb{C}_p$ such that
$$f^{k+1}(x_0)=\hat{x}.$$
Noticing that $f^{k+1}(x)=f^k(f(x))$, we must find a solution to the equation
$$f(x)=y.$$ By (3.1), we have
$$\dfrac{x^3+ax^2+bx}{ax+b}=y\Rightarrow{x^3}+ax^2+(b-ay)x-by=0.$$
Since $\mathbb{C}_p$ is algebraically closed, there exists an $x_0\in\mathbb{C}_p$ such that
$$x_0^3+ax_0^2+(b-ay)x_0-by=0\Rightarrow{f(x_0)=y}.$$
Thus, we have
$$f^{k+1}(x_0)=f^k(f(x_0))=f^k(y)=\hat{x}.$$
\item[2.]
We know that $|\hat x|_p={{\alpha^2}\over\delta}$. By the condition $\delta<\alpha$, we get $\alpha<{{\alpha^2}\over\delta}$.
By part 3 of Theorem \ref{t2}, for $x\in
S_{{{\alpha^2}\over\delta}}(0)$ and $x\neq \hat x$ , we have
$$\lim_{n\to\infty}|f^n(x)|_p=+\infty,$$
i.e., $S_{{{\alpha^2}\over\delta}}(0)\cap\mathcal P=\{\hat x\}=\mathcal
P_0$.
Denoting $r_0={{\alpha^2}\over\delta}$ we write $\mathcal P_0\subset S_{r_0}(0)$.

Now we find spheres containing the solutions to the equations $$f^k(x)=\hat x, \ \ k=1,2,3,... .$$
For each  $k$ we want to find some $r_k$ such that the solution $x$ of $f^k(x)=\hat x$, belongs to $S_{r_k}(0)$, i.e.,  $x\in S_{r_k}(0)$.
By Lemma \ref{lf2} we should have
$$\zeta_{\alpha,\delta}^k(r_k)={{\alpha^2}\over\delta}.$$
Now if we show that the last equation has a unique solution $r_k$ for each $k$, then
we get
$$\mathcal P_k=\{x\in \C_p: f^k(x)=\hat x\}\subset
S_{r_k}(0).$$

For such $r_k$, $k=0,1,2,...$ by definition of $\zeta_{\alpha,\delta}(r)$,  we have $\alpha<r_k\leq{{\alpha^2}\over\delta}$ and
$\zeta_{\alpha,\delta}(r_k)={{r_k^3}\over{\alpha^2}}.$
Thus $\zeta_{\alpha,\delta}^k(r_k)={{\alpha^2}\over\delta}$ has the form
$$\zeta_{\alpha,\delta}^k(r_k)={{r_k^{3^k}}\over{\alpha^{3^k-1}}}={{\alpha^2}\over\delta}$$
consequently
$$r^{3^k}_k=\alpha^{3^k}\cdot{\alpha\over\delta}.$$
Taking the $3^k$-root
we obtain the unique solution:
$r_k=\alpha\cdot\left({\alpha\over\delta}\right)^{1\over{3^k}}$.
\end{itemize}
\end{proof}

Let $\delta=\alpha=\beta$. We denote
$$\alpha'(x)=|f(x)|_p, \
\ {\rm for} \ \ x\in S_{\alpha}(0).$$
By Lemma \ref{lf2} and Lemma
\ref{l3} we get

\begin{thm}\label{t2} If $\delta=\alpha=\beta$, then
 the $p$-adic dynamical system generated by the function (\ref{fa}) has the following properties:
\begin{itemize}
\item[1.]
\begin{itemize}
\item[a.] $SI(x_0)=U_{\alpha}(0)$.
\item[b.] $\mathcal P\subset{S}_{\alpha}(0)$.
\end{itemize}
\item[2.] For $x\in{S}_{\alpha}(0)\setminus\mathcal P$:
\begin{itemize}
\item[a.] If $\alpha'(x)=\alpha$, then $f(x)\in{S}_{\alpha}(0)$.
\item[b.] If $\alpha'(x)<\alpha$, then $f^n(x)\in{S}_{\alpha'(x)}(0)$ for all $n\geq{1}.$
\end{itemize}
\item[3.] If $r\geq\alpha$ and $\alpha'(x)>\alpha$, then $\lim\limits_{n\rightarrow\infty}|f^n(x)|_p=+\infty.$
\end{itemize}
\end{thm}

\begin{proof} 1. a. By Lemma \ref{lf2} and part 1 of Lemma \ref{l3} we see that
 the spheres $S_r(0)$, where $r<\alpha$ are invariant for $f$.
 Consequently, $U_{\alpha}(0)\subset SI(x_0)$. Moreover, by Lemma \ref{lf2} and parts 2 and 3 of Lemma \ref{l3}, we know that $S_r(0)$ is not
an invariant of $f$ for $r\geq\alpha$. Thus $SI(x_0)=U_{\alpha}(0)$.

If $\delta=\alpha$, then we have $|\hat x|_p=\left|-{b\over a}\right|_p=\alpha$,
i.e., $\hat x\in S_{\alpha}(0)$. If $x\in S_r(0)$, $r\neq \alpha$,
then $f(x)\in S_r(0)$ for $r<\alpha$ or $f(x)\in S_{{r^2}\over\alpha}(0)$ for $r>\alpha$
and we have ${{r^2}\over\alpha}>r$ for all $r>\alpha$.
So $f(x)\in S_{\alpha}(0)$ iff $x\in S_{\alpha}(0)$.
Consequently, $\mathcal P\subset S_{\alpha}(0)$.

2. Let $x\in S_\alpha(0)\setminus\mathcal P$.

If $\alpha'(x)=\alpha$, then we have
$|f(x)|_p=\alpha$, i.e., $f(x)\in S_{\alpha}(0)$.

By part 1 of this Theorem if $\alpha'(x)<\alpha$, then $f^n(x)\in
S_{\alpha'(x)}(0)$ for any $n\geq 1$.

3. Follows from Lemma \ref{lf2} and part 3 of Lemma \ref{l3}.
\end{proof}

\section{Ergodicity of $f(x)$ in $\mathbb{Q}_p$.}

In this section we consider the dynamical system of the function (\ref{fa}) in
$\mathbb{Q}_p$ and we assume that the square root $\sqrt{a^2-4b}$ exists in $\mathbb{Q}_p$.

Define the following sets
$$A_1=\{r: \, 0<r<\alpha\} \ \ {\rm if} \ \ \delta\leq\alpha=\beta;$$
$$A_2=\{r: \, r\in(0,\beta)\setminus\{\alpha\}\} \ \ {\rm if} \ \
\alpha<\beta=\delta;$$ and we denote $A=A_1\cup A_2$ for
$\alpha\leq\beta$.

The next Corollary immediately follows from section 3 and the definition of $A$.

\begin{cor} The sphere $S_r(0)$ is invariant for $f$ if and only if $r\in A$.
\end{cor}
Since $x_0=0$ is an
indifferent fixed point, in this section we are
interested in studying ergodicity properties of the dynamical system.

\begin{lemma}{\label{ab}}
For every closed ball $V_{\r}(c)\subset S_r(0), \ {\rm where}
 \ r\in A$, the
following equality holds $$f(V_{\r}(c))=V_{\r}(f(c)).$$
\end{lemma}
\begin{proof}
Since $V_{\rho}(c)\subset{S}_r(0)$, $|c|_p=r$. Also, for any $x\in{V}_{\rho}(c)$, we have that
$|x-c|_p\leq\rho.$\\
We will show that $|f(x)-f(c)|_p=|x-c|_p\leq\rho.$ We now have that
$$|f(x)-f(c)|_p=|x-c|_p\dfrac{|(x+c)axc+(x^2+xc+c^2)b+a^2xc+b^2|_p}{|a|_p^2|x+\frac{x_1x_2}{a}|_p|c+\frac{x_1x_2}{a}|_p}.$$
Here, by properties of $p$-adic norms, we have that
$|(x+c)axc+(x^2+xc+c^2)b+a^2xc+b^2|_p=\max\{r^3\delta,r^2\alpha,\beta,r^2\delta,\alpha^2\beta^2\}$ by the distinction of norms,
and since $x\in{V}_{\rho}(c)\subset{S}_r(0)$, $|x|_p=r$.

It follows that
$$|(x+c)axc|_p\leq{r}^3\delta, \ |(x^2+xc+c^2)b|_p\leq{r}^2\alpha\beta, \ |a^2xc|_p=\delta^2r^2 \ {\rm and} \ |b^2|=\alpha^2\beta^2.$$
Case I: For $r\in{A}_1$, we have $\alpha<\beta=\delta$ and $r\in(0,\alpha)$ or $r\in(\alpha,\beta)$.
\begin{itemize}
\item If $r\in(0,\alpha)$, then $\max\{r^3\delta,r^2\alpha,\beta,r^2\delta,\alpha^2\beta^2\}=\alpha^2\beta^2$ which gives
$$|f(x)-f(c)|_p=|x-c|_p\dfrac{\alpha^2\beta^2}{\alpha^2\beta^2}=|x-c|_p\leq{\rho}.$$
\item If $r\in(\alpha,\beta)$, then $\max\{r^3\delta,r^2\alpha,\beta,r^2\delta,\alpha^2\beta^2\}=r^2\delta^2$. Thus,
$$|f(x)-f(c)|_p=|x-c|_p\dfrac{r^2\delta^2}{r^2\delta^2}=|x-c|_p\leq{\rho}.$$
\end{itemize}
Case II: For $r\in{A}_2$, we have $\delta\leq\alpha=\beta$ and $r\in(0,\alpha)$. It follows that
$$|f(x)-f(c)|_p=|x-c|_p\dfrac{\alpha^2\beta^2}{\delta^2\frac{\alpha^2\beta^2}{\delta^2}}=|x-c|_p\leq{\rho}.$$
Therefore,
$$f(V_{\rho}(c))=V_{\rho}(f(c)).$$
\end{proof}

\begin{lemma}\label{ab2}
If $c\in{S}_r(0)$ for $r\in{A}$, then
$$|f(c)-c|_p=\left\{\begin{array}{lllll}
\frac{r^3}{\alpha\beta}, \ \ {\rm if} \ \ 0<r<\alpha\\[2mm]
\frac{r^2}{\delta}, \ \ {\rm if} \ \ \alpha<r<\beta.
\end{array}
\right.$$
\end{lemma}
\begin{proof}
We first derive $|f(c)-c|_p$ to be
$$|f(c)-c|_p=\dfrac{|c|_p^3}{|a|_p|c+\frac{x_1x_2}{a}|_p}.$$
We now proceed by cases.

Case I: If $0<r<\alpha$, then $r\in{A}_1$ or $r\in{A}_2$. That is, $\alpha<\beta=\delta$ or $\delta\leq\alpha=\beta$ respectively.\\
In both scenarios, we have
$$|f(c)-c|_p=\dfrac{r^3}{\delta\frac{\alpha\beta}{\delta}}=\dfrac{r^3}{\alpha\beta}.$$

Case 2: If $\alpha<r<\beta$, then $r\in{A_2}$ which implies that $\alpha<\beta=\delta.$
It follows that
$$|f(c)-c|_p=\dfrac{r^3}{\delta{r}}=\dfrac{r^2}{\delta}.$$
\end{proof}

\begin{rk}\label{r2}
By Lemma \ref{ab2}, we see that $|f(c)-c|_p$ is dependent only on $r$.
 Thus, we may define the function $\rho:\mathbb{Q}_p\rightarrow[0,+\infty)$ to be
$$\rho(r)=|f(c)-c|_p.$$
\end{rk}

\begin{thm}\label{ca}
If $c\in{S}_r(0)$ for $r\in{A}$, then
\begin{itemize}
\item[1.] $|f^{n+1}(c)-f^n(c)|_p=\rho(r)$ for any $n\geq{1}$.
\item[2.] $f(V_{\rho(r)}(c))=V_{\rho(r)}(c).$
\item[3.] If $V_{\gamma}(c)\subset{S}_r(0)$ is an invariant for $f$, then $\gamma\geq{\rho(r)}.$
\end{itemize}
\end{thm}
\begin{proof}
Theorem \ref{ca} and its proof are coincidentally identical to
Theorem 11 and its corresponding proof in \cite{RS2}.
\end{proof}

For each $r\in A$ consider a measurable space $(S_r(0),\mathcal
B)$, where $\mathcal B$ is the algebra generated by closed
subsets of $S_r(0)$. Every element of $\mathcal B$ is a union of
some balls $V_{\r}(c)$.

A measure $\bar\mu:\mathcal B\rightarrow \mathbb{R}$ is said to be
the \emph{Haar measure} if it is defined by $\bar\mu(V_{\r}(c))=\r$.

Note that $S_r(a)=V_r(a)\setminus V_{r\over p}(a)$. So, we denote
$r'=\bar\mu(S_r(0))=r(1-{1\over p})$.

We consider the normalized (probability) Haar measure:
$$\mu(V_{\r}(c))={{\bar\mu(V_{\r}(c))}\over{\bar\mu(S_r(0))}}={{\r}\over{r'}}.$$

By Lemma \ref{ab} we conclude that $f$ preserves the measure
$\mu$, i.e.
\begin{equation}{\label{ab3}}
\mu(f(V_{\r}(c)))=\mu(V_{\r}(c)).
\end{equation}

Consider the dynamical system $(X,T,\mu)$, where $T:X\rightarrow
X$ is a measure preserving transformation and $\mu$ is a
probability measure. We say that the dynamical system is {\it
ergodic} if for every
invariant set $V$ we have $\mu(V)=0$ or $\mu(V)=1$ (see
\cite{Wal}).

\begin{thm}\label{t6}
\begin{itemize}
\item[1.] If $p\geq 3$, then the dynamical system $(S_r(0), f, \mu)$ is
not ergodic for any $r\in A$ (i.e. for any invariant sphere
$S_r(0)$).
\item[2.] Let $p=2$.
\begin{itemize}
\item[2.1)] If $\alpha=\beta$, then the dynamical system $(S_r(0), f, \mu)$ is not ergodic
for any $r\in A$.
\item[2.2)] If $\alpha<\beta$, then the dynamical system $(S_r(0), f, \mu)$ is not ergodic
for any $r\in A\setminus\{{\beta\over 2}\}$.
\end{itemize}
 Here
$\mu$ is the normalized Haar measure.
\end{itemize}
\end{thm}

\begin{proof}
Theorem \ref{t6} and its proof are coincidentally identical to Theorem 12 and its corresponding proof in \cite{RS2}.
\end{proof}

\begin{thm}{\cite{M}}{\label{erg3}}
Let $f,g: 1+2\mathbb Z_2\rightarrow 1+2\mathbb Z_2$ be
polynomials whose coefficients are $2$-adic integers.

Set $f(x)=\sum_ia_ix^i$, $g(x)=\sum_ib_ix^i$, and
$$A_1=\sum_{i \, {\rm odd}}a_i, \ \ A_2=\sum_{i \, {\rm even}}a_i, \ \ B_1=\sum_{i \, {\rm odd}}b_i, \ \
B_2=\sum_{i \, {\rm even}}b_i.$$

The rational function $R={f\over g}$ is ergodic if and only if one
of the following situations occurs:

(1) $A_1=1({\rm mod}\,4)$, $A_2=2({\rm mod} \, 4)$, $B_1=0({\rm mod} \, 4)$ and
$B_2=1({\rm mod} \, 4)$.

(2) $A_1=3({\rm mod} 4)$, $A_2=2({\rm mod} \, 4)$, $B_1=0({\rm mod} \, 4)$ and
$B_2=3({\rm mod} \, 4)$.

(3) $A_1=1({\rm mod} \, 4)$, $A_2=0({\rm mod} \, 4)$, $B_1=2({\rm mod} \, 4)$ and
$B_2=1({\rm mod} \, 4)$.

(4) $A_1=3({\rm mod} \, 4)$, $A_2=0({\rm mod} \, 4)$, $B_1=2({\rm mod} \, 4)$ and
$B_2=3({\rm mod} \, 4)$.

(5) One of the previous cases with $f$ and $g$ interchanged.
\end{thm}

We will now construct an example where $f(x)$ is ergodic for $p=2$ using this Theorem.

\begin{ex}
Recall that $\Z_2=\{x\in\Q_p: |x|_2\leq1\}$ and let $k\in\Z_2$. Letting
$$a=2k+\dfrac{1}{2} \ \ {\rm and} \ \ b=k$$
gives
$$f(x)=\dfrac{x^3+(2k+\frac{1}{2})x^2+kx}{(2k+\frac{1}{2})x+k}=\dfrac{2x^3+(4k+1)x^2+2kx}{(4k+1)x+2k}.$$
Notice that $\sqrt{a^2-4b}=2m-\frac{1}{2}\in\Q_p.$ In addition, $x_1=-2k$ and $x_2=-\frac{1}{2}$
which gives that
$$\alpha\leq\frac{1}{2} \ \ {\rm and} \ \ \beta=2.$$
Thus,
$$A=\{r:0<r<\alpha\leq\frac{1}{2} \ {\rm or} \ \alpha<r<2\}.$$
In relation to Theorem {\ref{erg3}}, we see that
$$A_1=2k+2, \ A_2=4k+1, \ B_1=4k+1 \ {\rm and} \ B_2=2k.$$
It follows that
$$A_1=0({\rm mod} \, 4) \ {\rm and} \ B_2=2({\rm mod} \, 4)$$
or
$$A_1=2({\rm mod} \, 4) \ {\rm and}  \ B_2=0({\rm mod} \, 4).$$
Since
$$B_1=1({\rm mod} \, 4) \ {\rm and} \ A_2=1({\rm mod} \, 4),$$
Theorem {\ref{erg3}} implies that $(S_1(0),f,\mu)$ is ergodic.
\end{ex}
We will now extend this to showing that if $4\alpha\leq\beta$ and $p=2$, then $(S_{\beta\over2}(0),f,\mu)$ is ergodic.\\
\\
Let a function $f(x): S_{p^l}(0)\rightarrow S_{p^l}(0)$ be given.
For $s\in\mathbb{Z}$ consider $x=g(t)=p^st$
$(t=g^{-1}(x)=p^{-s}x)$ then it is easy to see that
$f(g(t)):S_{p^{l+s}}(0)\rightarrow S_{p^l}(0)$. Consequently,
$g^{-1}(f(g(t))):S_{p^{l+s}}(0)\rightarrow S_{p^{l+s}}(0)$.

Let $\mathcal B$ (resp. $\mathcal B_s$) be the algebra generated by closed
subsets of $S_{p^l}(0)$ (resp. $S_{p^{l+s}}(0)$), and $\mu$ (resp. $\mu_s$)
 be the normalized Haar measure
on $\mathcal B$ (resp. $\mathcal B_s$).

Denote $(f\circ g)(t)=f(g(t)).$

\begin{thm}{\label{erg1}}\cite{RS2} The dynamical system
$(S_{p^l}(0), \, f, \, \mu)$ is ergodic iff
$(S_{p^{l+s}}(0), \, g^{-1}\circ f\circ g, \, \mu_s)$ is ergodic for some $s\in\mathbb{Z}$.
\end{thm}

\begin{thm}{\label{erg2}}
If $p=2$ and $4\alpha\leq\beta$, then the dynamical system
$(S_{\beta\over 2}(0), f, \mu)$ is ergodic.
\end{thm}

\begin{proof} Let $\alpha=|x_1|_2=2^q$ and $\beta=|x_2|_2=2^m$. Then
from $4\alpha\leq\beta$ we get $q\leq m-2$.
It is known that $x_1+x_2=-a$ and $x_1x_2=b$.
Consequently,  $a\in 2^{-m}(1+2\mathbb Z_2)$ and $b\in 2^{-q-m}(1+2\mathbb Z_2)$.

In $f:S_{2^{m-1}}(0)\rightarrow S_{2^{m-1}}(0)$ we change $x$ by
$x=g(t)=2^{1-m}t$. Then,
$|x|_2=2^{m-1}|t|_2=2^{m-1}$, $|t|_2=1$ and the function
$f(g(t)):S_1(0)\rightarrow S_{2^{m-1}}(0)$ has the following form
$$f(g(t))={{2^{3-3m}t^3+2^{2-2m}at^2+2^{1-m}bt}\over{2^{1-m}at+b}}.$$
Since $g^{-1}(f(g(t)))=2^{m-1}f(g(t))$ we have
$$|g^{-1}(f(g(t)))|_2=2^{1-m}|f(g(t))|_2=2^{1-m}2^{m-1}=1.$$
Thus, $g^{-1}(f(g(t))):S_1(0)\rightarrow S_1(0)$ and
\begin{equation}{\label{k}}
g^{-1}(f(g(t)))={{2t^3+2^mat^2+2^{2m-1}bt}\over{2^mat+2^{2m-1}b}}.
\end{equation}

For the numerator of (\ref{k}) we have $|2t^3|_2=2^{-1}$, $|2^mat^2|_2=1$ and
$|2^{2m-1}bt|_2=2^{q-m+1}\leq 2^{-1}$. Consequently,
$$2t^3+2^mat^2+2^{2m-1}bt=:\gamma_1(t), \ \ \mbox{is such that} \ \ \gamma_1: 1+2\mathbb Z_2\rightarrow 1+2\mathbb Z_2.$$
For the denominator (\ref{k}) we have $|2^mat|_2=1$ and
$|2^{2m-1}b|_2=2^{q-m+1}\leq 2^{-1}$. Therefore
$$2^mat+2^{2m-1}b=:\gamma_2(t)    \ \ \mbox{is such that} \ \  \gamma_2: 1+2\mathbb Z_2\rightarrow 1+2\mathbb Z_2.$$
Hence the function (\ref{k}) satisfies all conditions of Theorem \ref{erg3}.
Also, we note that if
$$2+2^{2m-1}b=2({\rm mod} \, 4), \ {\rm then} \ 2^{2m-1}=0({\rm mod} \, 4)$$
and vice versa.
In addition,
$$2^ma=1({\rm mod} \, 4).$$
Therefore using this theorem we conclude that the dynamical system
$(S_1(0), \, g^{-1}\circ f\circ g, \, \mu_{1-m})$ is ergodic. Consequently,  by Theorem \ref{erg1},
$(S_{2^{m-1}}(0), f, \mu)$, i.e.,  $(S_{\beta\over 2}(0), f,
\mu)$ is ergodic.
\end{proof}

\section{Uniformly local properties and periodic orbits of the rational function}

We will mainly be interested in the case where $X\subset \Q_p$ is a compact open subset of $\Q_p$. As before,
we assume $f : X \to X$ is well-defined and consider the dynamical system $(X,f,\mu)$.

\begin{defn}\label{d1}{\cite{DS}} Let $X$ be a compact open subset of $\Q_p$ and let $f : X \to X$ be a
rational function. Assume, in addition, that $f'(x)$ has no roots in $X$. Then $f$ is uniformly
locally scaling with local scalar $C(a) = |f'(a)|_p$; i.e., there exist $r > 0$ such that for any $a\in X$,
$|f(x)-f(y)|_p = |f'(a)|_p |x-y|_p$ whenever $x, y \in V_r(a)$.
\end{defn}

We can define the notions (uniformly) locally isometric, (uniformly) locally $\rho$-Lipschitz, (uniformly) locally bounded scaling as follows.

\begin{defn}\label{d2}{\cite{DS}} A map $f : X \to X$ is uniformly locally isometric (ULI) if there exists a constant $r > 0$
such that for all $a \in X$, $|f(x)-f(y)|_p=|x-y|_p$ whenever $x, y \in V_r(a)$.
\end{defn}
\begin{defn}\label{d3}{\cite{DS}} Let $\rho> 0$. A map $f : X \to X$ is uniformly locally $\rho$-Lipschitz (UL$\rho$L) if there exists a
constant $r > 0$ such that for all $a \in X$, $|f(x)-f(y)|_p\leq\rho|x-y|_p$ whenever $x, y \in V_r(a)$.
\end{defn}
\begin{defn}\label{d4}{\cite{DS}} A map $f : X \to X$ is uniformly locally bounded scaling (ULBS) if there exists a
constant $C > 0$ such that $f$ is uniformly locally scaling and $C(a)\leq C$ for all $a\in X$.
\end{defn}

By a similar argument as in Definition \ref{d1}, we have the following criteria.
\begin{pro}\label{p1}{\cite{DS}} Let $X$ be a compact open subset of $\Q_p$ and let $f : X \to X$ be a rational function.
Then
\begin{itemize}
\item[1.] $f$ is ULI if and only if $|f'(a)|_p = 1$ for all $a \in X$.
\item[2.] $f$ is UL$\rho$L if and only if $|f'(a)|_p\leq\rho$ for all $a\in X$.
\item[3.] $f$ is ULBS if and only if $f'(x)$ has no root in $X$.
\end{itemize}
\end{pro}
\begin{rk} It is easy to see that
$$ULI\Rightarrow UL\rho L\Rightarrow ULBS.$$

\end{rk}
\begin{thm}\label{t7}
Let $r\in A$ and suppose $x\in S_r(0)$ is arbitrary. Then
$$|f'(x)|_p=1.$$
\end{thm}

\begin{proof}
We see that
$$|f'(x)|_p=\big|\dfrac{2ax^3+3bx^2+a^2x^2+2abx+b^2}{a^2x^2+2abx+b^2}\big|_p.$$
Calculating the value of each norm, we have that
$$|2ax^3|_p=\left\{\begin{array}{ll}
\delta r^3 \ \ {\rm if} \ \ p\neq2\\[2mm]
\frac{1}{2}\delta r^3 \ \ {\rm if } \ \ p=2,
\end{array}
\right.
\ \ \ \ \ |3bx^2|_p=\left\{\begin{array}{ll}
\alpha\beta r^2 \ \ {\rm if} \ \ p\neq3\\[2mm]
\frac{1}{3}\alpha\beta r^2 \ \ {\rm if } \ \ p=3,
\end{array}
\right.
$$
$$|a^2x^2|_p=\delta^2r^2, \ \ \ \ \ |2abx|_p=\left\{\begin{array}{ll}
\delta\alpha\beta r \ \ {\rm if} \ \ p\neq2\\[2mm]
\frac{1}{2}\delta\alpha\beta r\ \ {\rm if } \ \ p=2,
\end{array}
\right. \ \
\mbox{and} \ \
|b^2|_p=\alpha^2\beta^2.$$
Since each norm is distinct for $r\in A$, it follows that
$$|2ax^3+3bx^2+a^2x^2+2abx+b^2|_p=\max\{|2ax^3|_p,|3bx^2|_p,|a^2x^2|_p,|2abx|_p,|b^2|_p\}$$
and
$$|a^2x^2+2abx+b^2|_p=\max\{|a^2x^2|_p,|2abx|_p,|b^2|_p\}.$$
Now, if $0<r<\alpha$ we have that
$$|f'(x)|_p=\dfrac{\alpha^2\beta^2}{\alpha^2\beta^2}=1$$
and if $\alpha<r<\beta$ then
$$|f'(x)|_p=\dfrac{\delta^2r^2}{\delta^2r^2}=1.$$
Thus, $|f'(x)|_p=1$ for all $x\in S_r(0)$.\\
\end{proof}

\begin{cor}\label{c4}
If $r\in A$ then $f:S_r(0)\rightarrow S_r(0)$ is uniformly locally isometric.
\end{cor}

\begin{cor}\label{c5}
If dynamical system $(S_r(0), f, \mu)$, $r\in A$ has $k$-periodic orbit $$y_0\to y_1\to ...\to y_k\to y_0$$ then $|(f^k(x))'_{x=y_i}|_p=1$ for all $i\in\{0,1,...,k\}$, i.e., character of periodic points is indifferent.
\end{cor}

Now suppose that $\{t_1, t_2\}$ is a $2$-periodic orbit in $S_r(0)$.

\begin{thm}\label{t9}
If $f(t_1)=t_2$ and $f(t_2)=t_1$, then for any $r>0$ we have $$f(S_r(t_1)\setminus\mathcal P)\subset S_r(t_2) \ \ \mbox{and} \ \ f(S_r(t_2)\setminus\mathcal P)\subset S_r(t_1).$$
\end{thm}

\begin{proof}
Suppose that $x\in S_r(t_1)$. That is, $|x-t_1|_p=r$. Now, by Corollary \ref{c4} we have
$$|f(x)-t_2|_p=|f(x)-f(t_1)|_p=|x-t_1|_p=r,$$
which implies that $f(x)\in S_r(t_2)$, i.e., $f(S_r(t_1)\setminus\mathcal P)\subset S_r(t_2)$.

By symmetry, we also have that $f(S_r(t_2)\setminus\mathcal P)\subset S_r(t_1)$ for any arbitrary $r$.
\end{proof}

We now consider the dynamical system $(S_r(0), \, f, \, \mu)$, when it is not ergodic.

We have that $x\in S_r(0)$ has the canonical form $x=p^{\gamma(x)}(x_0+x_1p+x_2p^2+...)$ and coefficients $a$ and $b$ have the canonical forms $$\ a=p^{\gamma(a)}(a_0+a_1p+a_2p^2+...) \ \ {\rm and} \ \ b=p^{\gamma(b)}(b_0+b_1p+b_2p^2+...).$$
In general, $f(x)$ has the canonical form
$$f(x)=x(1+\dfrac{x^2}{ax+b})=p^{\gamma(x)}(f_0+f_1p+f_2p^2+...).$$
Let $\dfrac{x^2}{ax+b}$ have the canonical form
$$\dfrac{x^2}{ax+b}=p^s(s_0+s_1p+s_2p^2+...).$$
\begin{pro}\label{p1}
If $x\in{S}_r(0)$ and $r\in{A}$, then for the function $f(x)=p^{\gamma(x)}(f_0+f_1p+f_2p^2+...),$
$f_i=x_i$ iff $i<s$.
\end{pro}
\begin{proof}
If $r\in{A}=A_1\cup{A}_2$, then $r\in{A}_1$ or $r\in{A}_2$.\\
Case I: Suppose $r\in{A_1}$, i.e., $\alpha<\beta=\delta$. Then $0<r<\alpha$ or $\alpha<r<\beta.$
\begin{itemize}
\item If $0<r<\alpha$, then $|x|_p^2=r^2<|ax+b|_p=\alpha\beta$. So $|\dfrac{x^2}{ax+b}|_p=p^{-s}<1$ and $s>0$.
\item If $\alpha<r<\beta$, then $|x|_p^2=r^2<|ax+b|_p=r\delta$. Again, we have that $|\dfrac{x^2}{ax+b}|_p=p^{-s}<1$ and $s>0$.
\end{itemize}
Case II: Suppose $r\in{A}_2$. Then, $\delta\leq\alpha=\beta$ and $0<r<\alpha$.\\
This gives that $|x|_p^2=r^2<|ax+b|_p=\alpha^2$ and $|\dfrac{x^2}{ax+b}|_p=p^{-s}<1.$ So $s>0.$\\
\\
In all cases, we conclude that $s>0$.\\
Since
$$\dfrac{x^2}{ax+b}=p^s(s_0+s_p+s_2p^2+...)$$
and
$$x=p^{\gamma(x)}(x_0+x_1p+x_2p^2+...),$$
we have that $$f(x)-x=x(1+\dfrac{x^2}{ax+b})-x=x\dfrac{x^2}{ax+b}=p^{\gamma(x)+s}(s_0+s_1p+s_2p^2+...)(x_0+x_1p+x_2p^2+...).$$
Since $(s_0+s_1p+s_2p^2+...),(x_0+x_1p+x_2p^2+...)\in\mathbb{Z}_p^*=\{x\in\mathbb{C}_p: \ \ |x|_p=1\}$, it follows that
$(s_0+s_1p+s_2p^2+...)(x_0+x_1p+x_2p^2+...)\in\mathbb{Z}_p^*.$\\
Denote
$$f(x)-x=p^{\gamma(x)+s}(t_0+t_1p+t_2p^2+...)=p^{\gamma(x)}(t_0p^s+t_1p^{s+1}+...),$$ where $t_0\neq 0$.

We also have that
$$f(x)-x=p^{\gamma(x)}((f_0-x_0)+(f_1-x_1)p+...+(f_{s-1}-x_{s-1})p^{s-1}+(f_s-x_s)p^{s}+...).$$
Now, combining both forms of $f(x)-x$ gives
$$p^{\gamma(x)}(0+0p+0p^2+...+0p^{s-1}+t_0p^s+t_1p^{s+1}+...)$$$$=p^{\gamma(x)}((f_0-x_0)+(f_1-x_1)p+...+(f_{s-1}-x_{s-1})p^{s-1}+(f_s-x_s)p^{s}+...).$$
Comparing coefficients of $p^i$ gives that $f_i-x_i=0$ if and only if $i<s$.
\end{proof}
For $x=p^{\gamma(x)}(x_0+x_1p+x_2p^2+...)$, we say that $x\in{V}_{r\over{p}}(c_k)$ if $x_0=k$, where $k=1,2,...,p-1$.\\
In addition, $$S_r(0)=\bigcup_{k=1}^{p-1}V_{r\over p}(c_k)$$ which gives that for any $x\in{S}_r(0)$, there exists a $k$ such that $x\in{V}_{r\over p}(c_k)$.\\
\\
The next Corollary is an immediate result of Proposition \ref{p1}.
\begin{cor}\label{c6}
The ball $V_{r\over p}(c_k)$, where $k=1,2,...,p-1$, is invariant for $f$.
\end{cor}
\begin{proof}
This follows from Proposition \ref{p1} and the fact that for $i=0$, $i<s$ which gives that $f_0=x_0$ for any $x_0\in\{1,2,...,p-1\}$.
\end{proof}

\begin{thm}\label{t10}
The ball $V_{r\over p^m}(c_k)$, where $k=1,2,...,p-1$ and $m\in\mathbb{N}$, is a minimal invariant ball if and only if one of the following holds:
\begin{itemize}
\item[1.] $r=p^{-{{\gamma(b)+m}\over2}}\in(0,\alpha)$
\item[2.] $r=p^{-\gamma(a)-m}\in(\alpha,\beta)$.
\end{itemize}
\end{thm}

\begin{proof}
By Theorem \ref{ca}, $V_{\rho(r)}(c)$ is the minimal invariant ball. Hence, a ball $V_{y}(c)$ is a minimum invariant ball iff $y=\rho(r).$\\
By Proposition \ref{1} and the canonical forms of $a$ and $b$, we have that
$$\alpha\beta=p^{-{\gamma(b)}} \ {\rm and} \ \delta=p^{-\gamma(a)}.$$
\begin{itemize}

\item[1.] By Lemma \ref{ab2}, if $0<r<\alpha$ then $\rho(r)=|f(c)-c|_p=\dfrac{r^3}{\alpha\beta}$. Then,

$$\rho(r)=\dfrac{r}{p^m}\Longleftrightarrow\dfrac{r^3}{\alpha\beta}=\dfrac{r}{p^m}\Longleftrightarrow{r^2}=\dfrac{\alpha\beta}{p^m}=p^{-\gamma(b)-m}\Longleftrightarrow{r}=p^{-{{\gamma(b)+m}\over2}}.$$
\item[2.] By Lemma \ref{ab2}, if $\alpha<r<\beta$ then $\rho(r)={r^2\over\beta}.$ Also, since $r\in{A}$, it must be the case that $\delta=\beta.$ Now,
$$\rho(r)=\dfrac{r}{p^m}\Longleftrightarrow\dfrac{r^2}{\delta}=\dfrac{r}{p^m}\Longleftrightarrow{r}=\dfrac{\delta}{p^m}=p^{-\gamma(a)-m}.$$
\end{itemize}
\end{proof}

We now fix $r\in A$ and consider the dynamical system generated by the function $f: S_r(0)\to S_r(0)$. We are interested in $2-$periodic trajectories of the dynamical system. Such trajectories are defined by solutions of the following equation
$$f^2(x)=x.$$
Solving this equations gives the polynomial
$$P(x)=x^6+3ax^5+(3b+4a^2)x^4+(7ab+2a^3)x^3+(3b^2+6a^2b)x^2+6ab^2x+2b^3=0.$$

\begin{lemma}\label{l4} Fix $r\in A$ and assume that the parameter $b\in S_r(0)$.
The polynomial $P(x)=0$ has a solution $x=b$ iff the parameter $a\in\mathbb{Q}_p\setminus\{0\}$ satisfies the following equality
$$b^3+(3a+3)b^2+(4a^2+7a+3)b+(2a^3+6a^2+6a+2)=0.$$
\end{lemma}

\begin{proof}
We see that
$$P(b)=b^6+3ab^5+(3b+4a^2)b^4+(7ab+2a^3)b^3+(3b^2+6a^2b)b^2+6ab^3+2b^3=0$$$$\Longleftrightarrow b^6+(3a+3)b^5+(4a^2+7a+3)b^4+(2a^3+6a^2+6a+2)b^3=0$$
$$\Longleftrightarrow b^3+(3a+3)b^2+(4a^2+7a+3)b+(2a^36a^2+6a+2)=0.$$
\end{proof}

\begin{rk}
A cubic equation may have not solutions in $\mathbb{Q}_p$. In papers \cite{MOS} and \cite{SA} some conditions of existence of such solutions are given. Therefore the condition of Lemma \ref{l4} can be checked by using the results of \cite{MOS} and \cite{SA}.
\end{rk}
By the Lemma we have

\begin{pro}
If $b\in S_r(0)$ for $r\in A$ then $|a|_p\neq1$.
\end{pro}
\begin{proof}
Since $b\in S_r(0)$, $|b|_p=r$. By Proposition 1, $r=\alpha\beta$ and $\delta=|a|_p$.\\
Suppose that $|a|_p=1$. Since $r\in A$, $r\in A_1$ or $r\in A_2$.\\
Case I: Suppose $r\in A_1$, then $\alpha<\beta=\delta.$ Since $\delta=|-a|_p=|a|_p=1$, we have that
$\beta=1$ and $r=\alpha\beta=\alpha$, hence $r\not\in A$, a contradiction.\\
Case II: Suppose $r\in A_2$, then $\delta\leq\alpha=\beta$. We now have that
$$1=|a|_p=\delta\leq\alpha=\beta\Rightarrow\alpha=\beta\leq\alpha\beta=r.$$
Thus, $r\not\in A$ a contradiction.\\
Therefore, $|a|_p\neq1$.
\end{proof}

We will now use techniques from \cite{ST} to parametrize rational solutions
to the cubic equation given in Lemma {\ref{l4}}. We let this cubic equation
 be given by $C$ as follows
$$C: b^3+(3a+3)b^2+(4a^2+7a+3)b+(2a^3+6a^2+6a+2)=0.$$
We are interested in finding rational points on this cubic curve.
Disregarding the condition that $b\neq 0$ for now gives that
$(b,a)=(0,-1)$ is a rational point on this cubic and in fact a
point of singularity (see \cite{ST}). Let
$$L: \ \ a=qb-1 $$
be the line of arbitrary rational slope $q$ that intersects the point.
The intersection of $L$ and $C$ will consist of the point
$(b,a)=(0,-1)$ and another point $(b_0,a_0)$ where $a_0,b_0\in\mathbb{Q}$.

Making the substitution $a=qb-1$ gives that
$$b^3+3b^3q+4b^3q^2+2b^3q^3-b^2q=0$$
which implies that
$$b=0 \ \ {\rm or} \ \ b=\dfrac{q}{1+3q+4q^2+2q^3}.$$
However, notice that if $b=0$ then $a=-1$. Denote
$$h(q)=\dfrac{q}{1+3q+4q^2+2q^3}.$$

Using the second representation of $b$, we now have parameters $b=h(q)$ and $a=qh(q)-1$ to rational points on the curve $C$.
Notice that $h(q)$ is undefined at $q=-1$.
Applying the conditions that $ab\neq0$ we see that $q\neq 0$ and $q\neq -\frac{1}{2}$.

We have now proved the following Theorem.
\begin{thm}\label{T13a} If
$$(b,a)\in \left\{\big(h(q), qh(q)-1\big): q\in\mathbb{Q}\setminus\{0,1,-\frac{1}{2}\}\right\}$$
then $ab\neq0$ and $P(b)=0.$
\end{thm}
\begin{rk}
Although we now have many rational solutions to $P(b)=0$, we must choose the parameter $q$ such that $|b|_p=r\in A$ and $\sqrt{a^2-4b}\in\mathbb{Q}_p$ to fully satisfy the conditions of Lemma {\ref{l4}}.
\end{rk}

\begin{pro}{\label{p6}} If the conditions of Lemma \ref{l4} are satisfied then the function $f$ has a 2-periodic orbit $\{b, f(b)\}$.
\end{pro}
We now use Theorem {\ref{T13a}} and Lemma {\ref{l4}} to construct an example of a 2-periodic orbit for $f$.
\begin{ex} Let $p=2$ and $q=1$. By Theorem {\ref{T13a}} this gives that
$b=\frac{1}{10}$ and $a=-\frac{9}{10}$ satisfy $P(b)=0$. Now,
$$|a|_2=|b|_2=2.$$
Notice that $\sqrt{a^2-4b}=\frac{\sqrt{41}}{10}$ exists in $\mathbb{Q}_2$ (see \cite{V}).

Since $b\in S_r(0)$ then $r=2$. Moreover  we have $\alpha=\beta=8$, $\delta=2$ and $r=2<\alpha=8$, therefore $r\in A$.
These conditions fulfill the requirements of Lemma {\ref{l4}}, so by Proposition {\ref{p6}}, $f$ has a 2-periodic orbit $\{\frac{1}{10},f(\frac{1}{10})\}.$
\end{ex}
\section*{ Acknowledgements}
The first author was supported by the National Science Foundation, grant number NSF HRD 1302873.

\end{document}